\documentclass[11pt,a4paper]{amsart}
\usepackage{amssymb,amsmath}
\usepackage{comment}

\usepackage[breaklinks]{hyperref}
\usepackage{breakurl}
\def\({\left(}
\def\){\right)}
\def\Nx{\nabla_x}
\def\Cal{\mathcal}

\def\eb{\varepsilon}

\def\R {\mathbb{R}}

\def\tilde{\widetilde}
\def\curl{\operatorname{curl}}
\newcommand{\be}{\begin{equation} }
\newcommand{\ee}{\end{equation} }

\def \and{\qquad\text{and}\qquad}

\def\Dt{\partial_t}
\def\Dx{\Delta}

\def\({\left(}
\def\){\right)}
\def\Nx{\nabla}
\def\divv{\operatorname{div}}
\def\eb{\varepsilon}
\def\Cal{\mathcal}
\def\eb{\varepsilon}

\def\R {\mathbb{R}}

\def\<{\left<}
\def\>{\right>}

\def \and{\qquad\text{and}\qquad}

\def\Dt{\partial_t}
\def\Dx{\Delta}

\newtheorem{proposition}{Proposition}[section]
\newtheorem{theorem}[proposition]{Theorem}
\newtheorem{corollary}[proposition]{Corollary}
\newtheorem{lemma}[proposition]{Lemma}
\theoremstyle{definition}
\newtheorem{definition}[proposition]{Definition}
\newtheorem{remark}[proposition]{Remark}

\numberwithin{equation}{section}
\def\be{\begin{equation}}
\def\ee{\end{equation}}
\def\bp{\begin{proof}}
\def\ep{\end{proof}}
\def \au {\rm}

\def \bk {\it}
\def \no#1#2#3 {{\bf #1} (#3), #2.}
\def \eds#1#2#3 {#1, #2, #3.}

\title[Determining modes depending on pressure]
{Restoring the Navier--Stokes dynamics by determining functionals depending on pressure}
\author[A. Ilyin, V. Kalantarov, A. Kostianko  and  S. Zelik]
{ Alexei Ilyin${}^{2,4}$, Varga Kalantarov${}^5$, Anna Kostianko${}^{1,2}$ and Sergey Zelik${}^{1,2,3,4}$}

\address{${}^1$ Zhejiang Normal University, Department of Mathematics, Zhejiang, China}
\address{${}^2$  HSE University, Nizhny Novgorod, Russia.}

\address{${}^3$ University of Surrey, Department of Mathematics, Guildford, UK.}

\address{${}^4$ Keldysh Institute of Applied Mathematics, Moscow, Russia}

\address{${}^5$ Department of mathematics, Ko{\c c} University, Istanbul, Turkey}

\email{ilyin@keldysh.ru}
\email{vkalantarov@ku.edu.tr}
\email{a.kostianko@imperial.ac.uk}
\email{s.zelik@surrey.ac.uk}

\begin{document}

\begin{abstract} For 2D Navier--Stokes equations in a bounded smooth domain, we construct a system of determining functionals which consists of $N$ linear continuous functionals which depend on pressure $p$ only and of one extra functional which is given by the value of vorticity at a fixed point $x_0\in\partial\Omega$.
\end{abstract}

\thanks{This work is  supported by   Russian
Science Foundation grant (project 23-71-30008)}
\maketitle
\tableofcontents

\section{Introduction}\label{s0}
It is expected that  the limit dynamics generated by dissipative PDEs is often essentially finite-dimensional and can be effectively described by finitely many parameters (the  order parameters in the terminology of I. Prigogine) governed by a system of ODEs describing its evolution in time (which is  referred as an Inertial Form (IF) of the system considered), see \cite{BV92,CV02,ha,MZ08,R01,SY02,T97,Z23} and references therein. A mathematically rigorous interpretation of this conjecture  leads to the concept of an Inertial Manifold (IM). By definition, IM is an invariant smooth  finite-dimensional  submanifold in the phase space  which is exponentially stable and possesses the  exponential tracking property, see \cite{CFNT89,FST88,M-PS88,M91,R94,Z14} and references therein. However, the existence of  an IM requires rather restrictive extra assumptions on the considered system, therefore, the IM may a priori not exist for many interesting equations arising in applications, see \cite{EKZ13,KZ18,KZ17,Z14} for more details. In particular, the existence or non-existence of an IM for 2D Navier--Stokes remains an open problem and for the simplified models of 1D coupled Burgers type systems such a manifold may not exist, see \cite{KZ18,KZ17,Z14}.
\par
By this reason, a number of weaker interpretations of the finite-di\-men\-sio\-na\-lity conjecture are intensively studied during the last 50 years.
One of them is related to the concept of a global attractor and interprets its finite-dimensionality in terms of Hausdorff or/and box-counting dimensions keeping in mind   the Man\'e projection theorem. Indeed, it is well-known that under some relatively weak assumptions
 the dissipative system considered   possesses a compact global attractor of  finite box-counting dimension, see \cite{BV92,MZ08,R01,T97} and references therein.  Then the  corresponding IF can be constructed   by projecting the considered dynamical system to a generic plane of sufficiently large, but finite dimension.  The key drawback of this approach is  the fact that the obtained IF is only H\"older continuous (which is not enough in general even for the uniqueness of solutions of the reduced system of ODEs). Moreover, as recent examples show, the limit dynamics may remain infinite-dimensional despite the existence of a H\"older continuous IF (at least, it may demonstrate some features, like super-exponentially attracting limit cycles, traveling waves in Fourier space, etc., which are impossible in classical dynamics), see \cite{Z14,Z23} for more details.
\par
An alternative, even  weaker approach (which actually has been historically the first, see \cite{FP67,Lad72}) is related to the concept of {\it determining functionals}. By definition, a system of continuous functionals $\Cal F:=\{F_1,\cdots,F_N\}$ on the phase space $H$ is called (asymptotically) determining if for any two trajectories $u_1(t)$ and $u_2(t)$, the convergence
$$
F_n(u_1(t))-F_n(u_2(t))\to0 \ \text{ as $t\to\infty$ for all $n=1,\cdots,N$}
$$
 implies that $u_1(t)\to u_2(t)$ in $H$ as $t\to\infty$.   Thus, in the case where $\Cal F$ exists, the limit behaviour of a trajectory $u(t)$ as $t\to\infty$ is uniquely determined by the behaviour of finitely many scalar quantities $F_1(u(t))$, $\cdots$, $F_N(u(t))$, see e.g. \cite{Cock1,Chu} for more details.
\par
To be more precise, it has been shown in \cite{FP67,Lad72} that, for the case of  Navier--Stokes equations in 2D, a system generated by  first $N$ Fourier modes is asymptotically determining if $N$ is large enough.
Later on the notions of determining nodes, volume elements, etc., have been introduced and various upper bounds for the number $N$ of elements in such systems have been obtained for various   dissipative PDEs (see, e.g., \cite{Chu1,FMRT, FT, FTT,JT,KT} and references therein). More general classes of determining functionals have been introduced in \cite{Cock1,Cock2}, see also \cite{Chu,Chu1}.
\par
Note that the existence of the finite system $\Cal F$ of determining functionals
{\it does not} imply in general that the quantities $F_n(u(t))$, $n\in1,\cdots,N$ obey
 a system of ODEs, therefore, this approach can not  a priori give the proper justification of   the finite-dimensionality conjecture. Moreover, these quantities usually obey only the delay differential equations (DDEs) whose phase space remains infinite-dimensional, see e.g. \cite{KKZ}. Nevertheless, they may be useful for many purposes, for instance, for establishing the controllability of an initially infinite dimensional system by finitely many modes (see e.g. \cite{AT14}), verifying the uniqueness of an invariant measure for random/stochasitc PDEs (see e.g. \cite{kuksin}), etc.  We also mention more recent but very promising applications of determining functionals to data assimilation problems where the values of functionals $F_n(u(t))$ are interpreted as the results of observations and the theory of determining functionals allows us to build new methods of restoring the trajectory $u(t)$ by the results of observations, see \cite{AT14,AT13,OT08,OT03} and references therein.
\par
It   is  worth  noting here that, despite the long history and large number of research papers and interesting results, the determining functionals are essentially less understood in comparison with IMs and global attractors. Indeed, the general strategy which most
 of the papers in the field follow on is to fix some very specific class of functionals (like Fourier modes or the values of $u$ in some prescribed points in  space (the so-called nodes), etc.) and to give as accurate as possible {\it upper} bounds for the number of such functionals which is sufficient to generate a determining system.
 These upper bounds are then expressed in terms of physical parameters of the considered system, e.g. in terms of the  Grashof number when 2D Navier--Stokes equations are considered, see \cite{Chu,OT08} for details.
  However, the described strategy  does not include verification of the sharpness of the obtained estimates, so, for a long time, very little was known about the proper lower
 bounds for the number $N$ of determining functionals.
 \par
  Moreover, the existing estimates often give the bounds for the number $N$ which are compatible to the box-counting dimension of the global attractor, so these one-sided bounds made an impression  that the number $N$ is a posteriori related with the dimension of the attractor and gives a realistic bound for the number of degrees of  freedom of the reduced system.
\par
On the other hand, this is clearly not so in the case of PDEs of one spatial dimension. Indeed, in this case, the minimal number $N$ is expressed in terms of the proper dimension of the set $\Cal R$ of {\it equilibria} of the considered system and, since any equilibrium solves an ODE, this dimension is naturally restricted by the order of the PDE considered and is unrelated with the size or the dimension of the attractor, see \cite{CT,kukavica,KKZ}.
\par
This "paradox" has been resolved in the recent paper \cite{KKZ}, where it was shown that the minimal number $N$ of determining functionals remains  controlled by the proper dimension of the equilibria set $\Cal R$ in the multi-dimensional case as well. The  difference is that, in multi-dimensional case, the equilibria solve an elliptic PDE, where, in contrast to the one-dimensional case, the dimension of $\Cal R$ may be essentially larger than the order of the system and is a priori restricted by the dimension of the global attractor only.  Thus, the "paradox" appears just because the existing methods of constructing determining functionals are not fine enough to distinguish the cases of big attractors and big equilibria sets, so more advanced methods should be invented (for instance, the results of \cite{KKZ} are obtained based on the proper version of Takens delay embedding theorem).  We also mention that, in many examples including 2D Navier--Stokes system, the set $\Cal R$ is finite for generic external forces and in this generic case $N=1$. Some drawback of this result is the fact that we do not know the explicit form of this functional, we only know that it can be chosen to be polynomial in $u$ and that the set of such functionals is generic (prevalent). In other words, arbitrarily small polynomial perturbation of any given functional makes it determining for 2D Navier--Stokes system.
\par
One more interesting question posed in the work \cite{KKZ} is the validity of the analogous results if we are allowed to use not all continuous functionals, but the functionals from some prescribed class only. For example, in the case of 2D Navier--Stokes equations, it would be natural (from both theoretical and practical points of view) to consider the functionals which depend on pressure $p$ only. To the best of our knowledge, there are no results in this direction available in the literature despite the fact that pressure is the most natural and easiest for implementation observable in hydrodynamics and  meteorology.
\par
The main aim of this paper is to give a partial answer to that question. Namely, we consider the 2D Navier--Stokes problem in a bounded smooth domain $\Omega$:
\begin{equation}\label{0.NS}
\Dt u+(u,\Nx)u+\Nx p=\nu\Dx u+g,\ \ \divv u=0,\ \ u\big|_{\partial\Omega}=0,
\end{equation}
where $u$ and $p$ are unknown velocity and pressure respectively, $\nu>0$ is a given kinematic viscosity and $g$ are external forces. The main result of the paper is verifying the existence of the system $\Cal F$ of determining functionals for this problem in the following form:
\begin{equation}\label{0.det}
\Cal F:=\{F_1(p), F_2(p),\cdots, F_N(p)\}\cup\{\curl u\big|_{x=x_0}\},
\end{equation}
where $x_0\in\partial\Omega$ is an arbitrary point on the boundary and the number $N$ is proportional to the fractal dimension of the corresponding global attractor, see Theorem \ref{Th3.1} for the rigorous statement.
\par
The paper is organized as follows. In Section \S\ref{s1}, we briefly recall the known facts about the solutions of the Navier--Stokes problem \eqref{0.NS}, the corresponding  global attractor $\Cal A$ and the estimates for its fractal dimension. We also discuss here how to restore the pressure $p$ by known velocity field $u$ and construct and study the pressure component $\Cal A_p$ of the global attractor which is crucial for the proof of our main result.
\par
In Section \S\ref{s2}, we discuss the theory of determining functionals. The most important for here are the facts that the system of functionals is determining if it is separating (i.e., if it separates different trajectories on the attractor), see \cite{KKZ} and the Man\'e projection theorem which allows us to construct a system of functionals $\{F_1(p),\cdots,F_N(p)\}$ with $N$ being proportional to the fractal dimension of $\Cal A$ which separate the pressure trajectories $p(t)$ on $\Cal A_p$.
\par
In Section \S\ref{s3} we state and prove our main result. Due to the results of previous sections, we only need to verify here that the pressure trajectory $p(t)$, $t\in\R$, which corresponds to the velocity trajectory $u(t)$, $t\in\R$ on the attractor, determines this trajectory in a unique way. i.e., that $p_1(t)\equiv p_2(t)$ for all $t$ implies the equality $u_1(t)=u_2(t)$ for all $t$. Unfortunately, up to the moment, we are unable to prove this statement in full generality (although, we know that it is true for some particular cases, say, for the case where $\Omega$ is a polygon, see \S\ref{s4}), so we have to add one more special functional, namely, the vorticity $\curl u$ of the solution $u$ computed at some fixed point $x_0\in\partial\Omega$. Based on this extra functional, we check that the difference $\bar u(t)=u_1(t)-u_2(t)$ satisfies the corresponding second order parabolic system with overdetermined boundary conditions $\bar u\big|_{\partial\Omega}=\Nx\bar u(t)\big|_{\partial\Omega}=0$ and the proper Carleman type estimate gives us the desired result.
\par
Finally, in Section \S\ref{s4}, we discuss some related questions and open problems which are raised up by our study and which have a substantial independent interest.

\section[Preliminaries I. 2D Navier--Stokes equations:\newline well-posedness, attractors and dimensions]{Preliminaries I. 2D Navier--Stokes equations:\\ well-posedness, attractors and dimensions}\label{s1}
The aim of this section is to recall the standard facts on the 2D Navier--Stokes equations which are necessary for what follows. Since all of them are well-known, we restrict ourselves by their statements, the details can be found e.g., in \cite{BV92,CV02,T97}, see also  references therein.
\par
In a bounded domain $\Omega$ of $\R^2$ with smooth boundary $\partial\Omega$, we consider the classical Navier--Stokes problem:
\begin{equation}\label{1.NS}
\begin{cases}
\Dt u+(u,\Nx)u+\Nx p=\nu\Dx u+g,\  u\big|_{\partial\Omega}=0,\ \  u\big|_{t=0}=u_0,\\
\divv u=0,
\end{cases}
\end{equation}
 where $u=(u_1(t,\mathtt x),u_2(t,\mathtt x))$ and $p=p(t,\mathtt x)$ are unknown velocity and pressure field respectively, $\nu>0$ and $g=(g_1(\mathtt x),g_2(\mathtt x))\in L^2(\Omega)$ are known kinematic viscosity and external forces respectively and $u_0=u_0(\mathtt x)$ is a given initial velocity field. We also recall that, as usual, the inertial term $(u,\Nx)u$ has the form
 $$
 (u,\Nx)u=u_1\partial_{x}u+u_2\partial_{y}u, \ \ \mathtt x=(x,y).
 $$
We recall that the pressure term $\Nx p$ can be removed applying the  Helmholtz--Leray projector $\Pi$ to equation \eqref{1.NS} which gives
\begin{equation}\label{1.Pro}
\Dt u+\Pi (u,\Nx)u=\nu Au+\Pi g,\ \ u\big|_{t=0}=u_0,\ \ u\big|_{\partial\Omega}=0,
\end{equation}
where $A:=\Pi\Dx$ is the Stokes operator. Vice versa, if the velocity field $u$ is known, the pressure term can be restored using the Poisson equation:
\begin{equation}\label{1.P}
\Dx p=-\divv(u,\Nx) u+\divv g,\ \ \partial_np\big|_{\partial\Omega}=\nu\Dx u\big|_{\partial\Omega}+g\big|_{\partial\Omega},
\end{equation}
which is formally obtained from \eqref{1.NS} by taking $\divv$, see e.g. \cite{T97}. Thus, it is enough to find the velocity field solving \eqref{1.Pro}, the pressure will be restored after that by solving \eqref{1.P} or otherwise.
\par
The standard phase space for problem \eqref{1.Pro} is the following one:
\begin{equation}\label{1.phase}
H:=\bigg\{u\in [L^2(\Omega)]^2:\, \divv u=0,\ u.n\big|_{\partial\Omega}=0\bigg\}.
\end{equation}
Namely, $u=u(t,\mathtt x)$ is a weak solution of \eqref{1.NS} if
\begin{equation}\label{1.sp}
u\in C(0,T;H)\cap L^2(0,T; H^1_0(\Omega))
\end{equation}
and for every divergent free test function $\varphi\in C_0^\infty([\R_+\times\Omega)$, we have the identity
\begin{equation}\label{1.weak}
-\int_\R (u,\Dt\varphi)\,dt+\nu\int_\R(\Nx u,\Nx\varphi)\,dt+\int_R((u,\Nx)u,\varphi)\,dt=\int_R(g,\varphi)\,dt.
\end{equation}
Here and below $H^s(\Omega)$ (resp. $H^s_0(\Omega)$) stands for the standard Sobolev space of distributions whose derivatives up to order $s$ belong to the Lebesgue space $L^2(\Omega)$ (resp. the closure of $C_0^\infty(\Omega)$ in the norm of $H^s(\Omega)$) and $(u,v):=\int_\Omega u(\mathtt x).v(\mathtt x)\,d\mathtt x$.
\par
The following theorem is one of the main results in the theory of 2D Navier--Stokes equations.
\begin{theorem}\label{Th1.wp} Let $g\in L^2(\Omega)$ and $u_0\in H$. Then, for any $T>0$,  problem \eqref{1.NS} possesses a unique weak solution $u(t)$ and therefore the solution semigroup
$$
S(t): H\to H,\ \ S(t)u_0:=u(t)
$$
is well-defined in the phase space $H$. Moreover, this semigroup is dissipative, i.e.,
\begin{equation}
\|S(t)u_0\|_H\le Q(\|u_0\|_H)e^{-\alpha t}+Q(\|g\|_{L^2}),
\end{equation}
where $\alpha>0$ and the monotone function $Q$ are independent of $t$, $u_0$ and $g$. Furthermore, this semigroup possesses the parabolic smoothing property, i.e. $S(t)u_0\in H^2(\Omega)$ for all $t>0$ and
\begin{multline}\label{1.dif}
\|S(t)u_0\|_{H^2}\le Q(1/t+\|u_0\|_{H}),\\ \|S(t)u_0^1-S(t)u_0^2\|_{H^2}\le Q(1/t+\|u_0^1\|_H^2+\|u_0^2\|_{H^2})\|u_0^1-u_0^2\|_H
\end{multline}
for $t\in(0,1]$ and
some  monotone function $Q$.
\end{theorem}
For the proof of this theorem, see e.g. \cite{BV92,T97}.
\par
This result allows us to apply the theory of attractors to study the asymptotic behavior of the solution semigroup $S(t):H\to H$. We recall that a set $\Cal A\subset H$ is called a global attractor for the semigroup $S(t)$ if
\par
a) $\Cal A$ is a compact set in $H$;
\par
b) $\Cal A$ is strictly invariant, i.e. $S(t)\Cal A=\Cal A$ for all $t\ge0$;
\par
c) $\Cal A$ attracts the images of all bounded sets as $t\to\infty$, i.e., for every neighbourhood $\Cal O(\Cal A)$ of $\Cal A$ in $H$ and every bounded set $B\subset H$ there exists $T=T(\Cal O,B)$ such that
$$
S(t)B\subset \Cal O(\Cal A)
$$
for all $t\ge T$.
\par
As a corollary of Theorem \ref{Th1.wp} we get.
\begin{corollary}\label{Cor1.attr} Under the assumptions of Theorem \ref{Th1.wp}, the solution semigroup $S(t):H\to H$ associated with Navier--Stokes equation \eqref{1.NS} possesses a global attractor $\Cal A$ which is a compact set in $H\cap H^1_0(\Omega)\cap H^2(\Omega)$. Moreover, the attractor $\Cal A$ is generated by the set $\Cal K\subset C_b(\R,H)$ of all bounded solutions of \eqref{1.Pro} which are defined for all $t\in\R$:
$$
\Cal A=\Cal K\big|_{t=0}.
$$
\end{corollary}
The next important fact is that the attractor $\Cal A$ has the finite fractal dimension. We recall that if $\Cal A$ is a compact set in $H$ then, by the Hausdorff criterion, for every $\eb>0$, it can be covered by finitely many $\eb$-balls in $H$. Denote by $N_\eb(\Cal A,H)$ the minimal number of such balls. Then, the fractal dimension $d_f(\Cal A,H)$ is defined via
$$
d_f(\Cal A,H):=\limsup_{\eb\to0}\frac{\log_2N_\eb(\Cal A,H)}{\log_2\frac1\eb}.
$$
\begin{theorem}\label{Th1.dim} Under the assumptions of Theorem \ref{Th1.wp}, the fractal dimension of the attractor $\Cal A$ is finite:
\begin{equation}\label{1.attr-fin}
\dim_f(\Cal A,H)<\infty.
\end{equation}
Moreover, the fractal dimension of $\Cal A$ in $H^2(\Omega)$ is also finite and coincides with its dimension in $H$:
\begin{equation}
\dim_f(\Cal A,H^2(\Omega))=\dim_f(\Cal A,H)<\infty.
\end{equation}
\end{theorem}
The proof of this theorem can be found e.g. in \cite{T97}. We recall that coincidence of the dimensions in $H$ and $H^2(\Omega)$ follows from the fact that the solution operators $S(t): H\to H^2(\Omega)$ are globally Lipschitz continuous on the attractor for all $t>0$.
\begin{remark} The modern  explicit estimate for the dimension of the attractor $\Cal A$ is given by
\begin{equation}\label{1.a-best}
\dim_f(\Cal A,H)\le \frac{C_{LT}^{1/2}}{2\sqrt2\pi} G,
\end{equation}
where $G:=\frac{\|g\|_{L^2}\operatorname{vol}(\Omega)}{\nu^2}$ is the Grashof number  and $C_{LT}$ is the constant in the appropriate Lieb-Thirring inequality, see e.g. \cite{IKZ-s}. The best up to the moment upper bound for $C_{LT}$ is $C_{LT}\le \frac{R}{2\pi}$, where $R\sim 1.456$, see \cite{FR,FRL}.
\end{remark}
We now return to equation \eqref{1.P} for recovering pressure from velocity. First, without loss of generality, we may assume that
\begin{equation}\label{1.norm}
<p>:=\frac1{\operatorname{vol}(\Omega)}\int_\Omega p(t,\mathtt x)\,d\mathtt x\equiv 0,\ \ \divv g=0.
\end{equation}
Then, since $\Dx u$ and $g$ are divergent free, by the embedding theorem, the boundary data $\partial_n p\big|_{\partial\Omega}$ is well-defined and belongs to $H^{-1/2}(\partial\Omega)$, see \cite{T83} if $u\in H^2(\Omega)\cap H$. Since $H^2(\Omega)\subset C$, $(u,\Nx)u\in H^1(\Omega)$ and $\divv ((u,\Nx)u)\in L^2(\Omega)$, so the Neumann boundary value problem \eqref{1.P} is well-posed in $H^1(\Omega)\cap \{<p>=0\}$. Let us denote by $p=P(u)$ the unique solution of this problem (which has zero mean). Then
\begin{equation}\label{1.2P}
P: H^2(\Omega)\cap H\to H^1(\Omega)
\end{equation}
is a smooth quadratic map. We denote by $\Cal A_p:=P(\Cal A)$ the image of the attractor $\Cal A$ under this map. The next simple reult is, however, crucial for what follows
\begin{corollary}\label{Cor1.p-dim} Under the above assumptions, we have
\begin{equation}
\dim_f(\Cal A_p, H^1(\Omega))\le \dim_f(\Cal A,H^2(\Omega))=\dim_f(\Cal A,H).
\end{equation}
\end{corollary}
\section{Preliminaries II. Determining and separating functionals}\label{s2}
We now turn to determining functional and start with the necessary definitions, see \cite{Chu,Chu1,FT,FTT,KKZ} and references therein for more details.
\begin{definition}\label{Def2.det} Let $S(t):H\to H$, $t\ge0$, be a semigroup acting in a metric space $H$. Then, a system $\Cal F=\{F_1,\cdots,F_n\}$ of continuous functionals $F_i:H\to\R$ is called asymptotically determining if for any two trajectories $u_i(t):=S(t)u_i^0$, $i=1,2$, the convergence
$$
\lim_{t\to\infty}(F_i(u_1(t))-F(u_2(t)))=0,\ \ i=1,\cdots,n
$$
implies that $\lim_{t\to\infty}d(u_1(t),u_2(t))=0$.
\par
Assume in addition that $S(t)$ possesses a global attractor $\Cal A$ in $H$. Then a system $\Cal F$ is called separating if for any two complete trajectories $u_1(t)$ and $u_2(t)$, $t\in\R$, belonging to the attractor, the equalities
\begin{equation}\label{2.sep}
F_i(u_1(t))\equiv F_i(u_2(t)),\ \ \text{for all $t\in\R$ and $i=1,\cdots,n$}
\end{equation}
imply that $u_1(t)\equiv u_2(t)$.
\end{definition}
\begin{proposition}\label{Prop2.sd} Let the semigroup $S(t):H\to H$ possess a global attractor $\Cal A$ in $H$ and let the operators $S(t)$ be continuous for every $t\ge0$. Then any separating system $\Cal F=\{F_1,\cdots,F_n\}$ of functionals is asymptotically determining.
\end{proposition}
For the proof of this result see e.g. \cite{KKZ}. The example of an asymptotically determining system which is not separating is also given there.
\begin{remark} Although the notion of an asymptotically determining functionals is formally a bit weaker than the separating functionals, the difference does not look essential and the existing examples/counterexamples look rather artificial. On the other hand, verifying the separation property is substantially simpler since the solutions on the attractor usually have better regularity properties than arbitrary semi-trajectories. In addition, we may use arguments related with Carleman estimates for that purposes which is crucial for what follows. For these reasons, we will consider only separating functionals in the sequel.
\end{remark}
\begin{remark} Recall that separating functionals should separate the equilibria $\Cal R\subset\Cal A$ of the semigroup $S(t)$. This gives us a natural lower bound for the number $n$ in any separating or asymptotically determining system:
\begin{equation}\label{2.emb}
n\ge \dim_{emb}(\Cal R,H),
\end{equation}
where the embedding dimension on the right-hand side is, by definition, the least number $d\in\mathbb N$ such that there exists a continuous embedding of $\Cal R$ to $\R^d$. In particular, for many cases, the set $\Cal R$ of equilibria is generically discrete, so its embedding dimension is one. More surprising is that, for wide classes of PDEs, generically this estimate is sharp, so the properly chosen single  functional is separating/asymptotically determining. This result is based on the Takens delay embedding theorem and is proved in \cite{KKZ}.
\end{remark}
Unfortunately, we are not able to verify such a result for functionals depending on pressure only, so we have to proceed in a more standard way using the fractal dimension of a global attractor as a natural upper bound. This approach is based on the Man\'e projection theorem.
\begin{theorem} Let $H$ be a Hilbert space and let $\Cal B$ be a compact subset of $H$ such that $\dim_f(\Cal B,H)<d$. Then, for a generic (prevalent) set of $2d$-dimensional subspaces $V$ of $H$, the orthogonal projector $P_V$ to the subspace $V$ is one-to-one on $\Cal B$.
\end{theorem}
For the proof of this classical result and related topics, see \cite{R11}.
\par
Let $V$ be such a generic subspace with base $\{e_k\}_{k=1}^{2d}$. Then
$$
P_Vu=\sum_{k=1}^{2d}(u,e_k)e_k
$$
and the system of linear functionals $F_k(u):=(u,e_k)$ allow us to distinguish  points of $\Cal B$ (in the sense that $F_k(u_1)=F_k(u_2)$ for all $k=1,\cdots,2d$ if and only if $u_1=u_2$). This gives us the following simple, but important result.
\begin{corollary}\label{Cor2.sep} Let the semigroup $S(t):H\to H$ acting on a Hilbert space $H$ possess a global attractor whose fractal dimension $\dim_f(\Cal A,H)<d$. Let $V$ be the $2d$ dimensional subspace of $H$ such that the corresponding orthoprojector is one-to-one on the attractor and let $\{e_k\}_{k=1}^{2d}$ be the orthonormal base of $V$. Then the system of functionals  $\Cal F=\{F_i\}_{i=1}^{2d}$ defined by $F_i(u):=(u,e_i)$ is separating on the attractor.
\end{corollary}
Note that, in contrast to determining ones,  the concept of separating on $\Cal B$ ( where $\Cal B$ is a compact set) functionals requires only that the functions $u(t)\in \Cal B$ for all $t\in\R$, and does not require them to be the trajectories of a semigroup. We utilize this observation in the following way.
\begin{definition}\label{Def2.p-sep} Let the assumptions of Theorem \ref{Th1.dim} hold and let $\Cal A$ be the global attractor of the solution semigroup $S(t):H\to H$ associated with the Navier--Stokes equation. Let also $\Cal A_p$ be its pressure component defined via $\Cal A_p:=P(\Cal A)$, see \eqref{1.2P} and let $\Cal K_p\subset C_{loc}(\R,\Cal A_p)$ be the set of pressure components of all solutions, belonging to the attractor $\Cal A$:
\begin{multline}
\Cal K_p:=\{p(t)\in \Cal A_p,\ t\in\R\,: \exists u(t)\in\Cal A, \ t\in\R,\ u\ \text{is a complete}\\\text{ trajectory of the solution semigroup $S(t)$}\}.
\end{multline}
We say that the system $\Cal F=\{F_1,\cdots,F_n\}$, $F_i:\Cal A_p$, $i=1,\cdots,n$ is pressure separating on the attractor if for every two $p_1,p_2\in\Cal K_p$, the equalities
\begin{equation}\label{2.psep}
F_i(p_1(t))\equiv F_i(p_2(t)), \ i=1,\cdots,n, \ t\in\R,
\end{equation}
imply that $p_1(t)\equiv p_2(t)$ for all $t\in\R$.
\end{definition}
Then, analogously to Corollary \ref{Cor2.sep}, we have the following result.
\begin{corollary}\label{Cor2.p-sep} Let the assumptions of Corollary \ref{Cor1.p-dim} hold and let
$$
\dim_f(\Cal A_p,H^1(\Omega))<d.
$$
 Then there exists a system of pressure separating functionals $\Cal F$ with the number of functionals which does not exceed $2d$.
\end{corollary}
\begin{remark} For a given system $\Cal F$ of   functionals depending on pressure, we may construct a system of functionals $\tilde {\Cal F}$, by the following natural formula: $\tilde F_i(u):=F_i(P(u))$. If this system $\tilde{\Cal F}$ is separating on the attractor $\Cal A$, then $\Cal F$ is obviously pressure separating. The opposite statement is a priori not true since we do not know whether the one-to-one correspondence between pressure and velocity components holds.
\end{remark}

\section{Main result}\label{s3}
We are now ready to state and prove our main result concerning the determining functionals depending on pressure. We start with a system of pressure separating functionals $\Cal F=\{F_1(p),\cdots, F_{2d}(p)\}$ constructed in Corollary \ref{Cor2.p-sep} and the corresponding lifting $\tilde{\Cal F}$ to the velocity phase space. Unfortunately, we do not know whether or not such a system will be separating on the attractor $\Cal A$ of the Navier--Stokes problem considered, so we need to add one more specific functional to this system related with the value of vorticity at some point on the boundary and then the constructed system will be separating and, therefore, asymptotically determining. Namely,
let us fix a solution $G\in H\cap H^2$ of the stationary Navier--Stokes problem
\begin{equation}\label{3.ST}
(G,\Nx)G+\Nx q=\nu\Dx G+g,\ \ \divv G=0,\ \ G\big|_{\partial\Omega}=0
\end{equation}
and a point $\mathtt x_0\in\partial\Omega$ and consider the functional
\begin{equation}\label{3.vor}
\Omega(u):=(\curl u-\curl G)\big|_{\mathtt x=\mathtt x_0},\ \ \curl V:=\partial_{y}V_1-\partial_{x}V_2.
\end{equation}
\begin{remark} Note that we do not have enough regularity of the solution $u(t)$ (even belonging to the attractor) in order to define the trace of the vorticity at a point on the boundary. Indeed, if $g\in L^2(\Omega)$, we can prove only that $u(t)\in H^2(\Omega)$, so $\curl u(t)\in H^1(\Omega)$ which is not embedded to $C(\bar\Omega)$. But the difference $u(t)-G$ is more regular for all $t>0$, say, $u(t)-G\in H^3(\Omega)$, so $\curl(u-G)\in H^2(\Omega)\subset C(\bar\Omega)$, so the functional \eqref{3.vor} is well-defined. This is the only reason to introduce the function $G$. Alternatively, since $\curl u\big|_{\partial\Omega}\in L^p(\partial\Omega)$ for all $p<\infty$, we may use some mean values of it on the boundary instead of subtracting the solution of the stationary equation.
\end{remark}
 The following theorem can be considered as the main result of the paper.
\begin{theorem}\label{Th3.1} Let the assumptions of Corollary \ref{Cor2.p-sep} be satisfied and let
$\Cal F=\{F_1(p),\cdots,F_{2d}(p)\}$ be the system of pressure separating functionals constructed there. Then, the system of functionals
\begin{equation}\label{3.det}
\bar{\Cal F}:=\{F_1(P(u)),\cdots,F_{2d}(P(u)),\Omega(u)\}
\end{equation}
is separating on the global attractor $\Cal A$ of the initial Navier--Stokes problem and, therefore, is asymptotically determining.
\end{theorem}
\begin{proof} Let $u_1(t)$ and $u_2(t)$ be two solutions of the Navier--Stokes problem such that $F(u_1(t))\equiv F(u_2(t))$ for all $t\in\R$ and $F\in \bar{\Cal F}$. Then, since $\Cal F$ is pressure separating, we conclude that $p_1(t)\equiv p_2(t)$ for all $t\in\R$.
\par
Let $\bar u(t):=u_1(t)-u_2(t)$. Then this function solves the following  overdetermined  system of semilinear heat equations
\begin{equation}\label{3.3}
\Dt\bar u-\Dx\bar u+(u_1,\Nx)\bar u+(\bar u,\Nx) u_2=0, \ \bar u\big|_{\partial\Omega}=0, \ \divv \bar u=0
\end{equation}
for all $t\in\R$. Moreover, due to our extra vorticity functional, we also know that
\begin{equation}\label{3.vor0}
\curl\bar u(t,\mathtt x_0)\equiv0
\end{equation}
The next lemma is the key tool in the proof of Theorem \ref{Th3.1}. We assume for simplicity that $\Omega$ is simply connected, so the boundary $\partial\Omega$ is a closed smooth curve without intersections $\gamma:=\{(x(s),y(s)),\ s\in[0,L]\}$, where $s$ is a natural parameter. In a general case, we just need to consider the  connected component of the boundary which contains the point $\mathtt x_0$. We also need to fix the orientation. for definiteness, we do in such way that $\gamma(s)$ move clock-wise along $\partial\Omega$ when $s$ grows.
\begin{lemma}\label{Lem1.2} Under the assumptions of the theorem,
\begin{equation}\label{3.4}
\bar u(t)\big|_{\partial\Omega}=\partial_n\bar u\big|_{\partial\Omega}=0
\end{equation}
for all $t\in\R$.
\end{lemma}
\begin{proof}[Proof of the lemma] We first note that, due to the parabolic smoothing for the linearized Navier--Stokes equations, $\bar u(t)\in H^3(\Omega)$ for all $t\in\R$, so all traces of the second derivatives on the boundary involved in the proof below  make sense.
\par
The idea of the proof is to derive differential equations for the normal derivatives
 $\partial_n \bar u(s)$  and to verify that these equations possess only zero solution. To this end, we need to remind some standard formulas for tangential and normal derivatives. Let $U$ be a smooth function in $\bar\Omega$. Then, obviously
 \begin{equation}\label{3.5}
 \partial_\tau U\big|_{\partial\Omega}=\frac d{ds}U(x(s),y(s))=x'\partial_x U+y'\partial_y U,
 \ \partial_n U=-y'\partial_x U+x'\partial_y U.
 \end{equation}
 In the last formula we have implicitly used that $n(s):=(-y'(s),x'(s))$ is a unit outer normal vector to the boundary. We now compute the tangential derivatives of quantities \eqref{3.5} using that 
 $$
 n'(s)=(-y''(s),x''(s))=\kappa(s)(x'(s),y'(s)),
 $$
  where $\kappa(s)$ is the curvature of the curve $\gamma(s)$. Namely,
\begin{multline}\label{3.6}
 \partial_\tau(\partial_\tau U)=\frac d{ds}(x'\partial_x U+y'\partial_yU)=\\=
 x''\partial_x U+y''\partial_y U+
 (x')^2\partial_{xx}U+2x'y'\partial_{xy}U+(y')^2\partial_{yy}U=\\=
 (x')^2\partial_{xx}U+2x'y'\partial_{xy}U+(y')^2\partial_{yy}U-\kappa(s)\partial_nU.
 \end{multline}
Analogously,
\begin{multline}\label{3.7}
\partial_\tau(\partial_n U)=\frac d{ds}(-y'\partial_xU+x'\partial_yU)=\\=\kappa(s)\partial_\tau U+x'y'(\partial_{yy}U-\partial_{xx}U)+((x')^2-(y')^2)\partial_{xy}U.
\end{multline}
We now return to equations \eqref{3.3}. Let $\bar u=(v,w)$. Then, we have 6 equations for determining
the traces of the second derivatives of $v$ and $w$ on the boundary:
\begin{equation}\label{3.8}
\Dx v\big|_{\partial\Omega}=\Dx w\big|_{\partial\Omega}=\partial_x\divv\bar u\big|_{\partial\Omega}=\partial_y\divv\bar u\big|_{\partial\Omega}=\partial_{\tau,\tau}v=\partial_{\tau,\tau}w=0
\end{equation}
for all $t\in\R$. The first two equations are obtained from \eqref{3.3} and the rest two are due to Dirichlet boundary conditions. Let us express all second derivatives through the mixed derivatives $A=\partial_{xy}v$ and $B=\partial_{xy}w$ of $v$ and $w$ respectively using the first 4 equations of \eqref{3.8}. This gives
$$
\partial_{xx}v=-\partial_{xy}w=-B,\ \partial_{yy}v=B,\ \partial_{xx}w=A,\ \ \partial_{yy}w=-A
$$
Inserting these relations in the last two equations of \eqref{3.8} and using \eqref{3.6}, we end up with the following system on $A$ and $B$:
\begin{multline}\label{3.9}
((y')^2-(x')^2)B+2x'y' A=\kappa(s)\partial_nv,\\  2x'y'B-((y')^2-(x')^2)A=\kappa(s)\partial_nw.
\end{multline}
We see that the determinant of this system is non-zero ($-\det=((x')^2+(y')^2)^2=1$), therefore all second derivatives on the boundary can be expressed in terms of first normal derivatives of $v$ and $w$. We are specially interested in such expressions for $\partial_\tau(\partial_n\bar u)$ which will allow us to get the closed system for determining first normal derivatives. Using formula \eqref{3.7} and the fact that tangential derivatives vanish, we get
$$
\partial_\tau(\partial_n v)=2x'y'B-((y')^2-(x')^2)A=\kappa(s)\partial_nw
$$
and, analogously,
$$
\partial_\tau(\partial_nw)=-2x'y'A+((x')^2-(y')^2)B=-\kappa(s)\partial_nv.
$$
Thus, the functions $X(s):=\partial_nv\big|_{\partial\Omega}$ and $Y(s):=\partial_nw\big|_{\partial\Omega}$ satisfy the system
$$
X'(s)=\kappa(s)Y(s),\ \ Y'(s)=-\kappa(s)X(s).
$$
Note that these equations coincide with the Frenet-Sarret equations for $x'(s)$ and $y'(s)$ and has a general solution in the following form
\begin{multline}\label{3.10}
X(s)=R\cos(\Theta(s)+\theta_0),\ \ Y(s)=R\sin(\Theta(s)+\theta_0),\\ \Theta(s):=\int_0^s\kappa(\tau)\,d\tau.
\end{multline}
Moreover, we have the analogous formulas for $x'(s)$ and $y'(s)$:
\begin{equation}\label{3.11}
x'(s)=\cos(\Theta(s)+\theta_0'),\ \ y'(s)=\sin(\Theta(s)+\theta_0').
\end{equation}
In order to determine the phases, we look at the divergence free condition. Since $\partial_\tau U\big|_{\partial\Omega}=0$, we have $\Nx U\big|_{\partial\Omega}=\partial_n U n$ and therefore (see \eqref{3.5}),
$$
\partial_x U=-y'(s)\partial_n U,\ \ \partial_y U=x'(s)\partial_nU.
$$
Then
\begin{multline}
\partial_x v+\partial_yw=-y'\partial_nv+x'\partial_n w=-R\cos(\Theta+\theta_0)\sin(\Theta+\theta_0')+\\+
R\cos(\Theta+\theta_0')\sin(\Theta+\theta_0)=R\sin(\theta_0-\theta_0')=0.
\end{multline}
Thus, the vectors $(X(s),Y(s))$ and $\tau(s)=(x'(s),y'(s))$ must be parallel, so we  get
$$
X(s)=Rx'(s),\ \ Y(s)=Ry'(s).
$$
Let us now compute the vorticity $\omega(s):=\partial_{y}v-\partial_xw$ at the boundary:
$$
\omega(s)=x'(s)Rx'(s)+y'(s)Ry'(s)=R((x'(s))^2+(y'(s))^2)=R.
$$
Thus, the trace of the vorticity of $\bar u$ at the boundary is independent of $s$ and we have found that
$$
\partial_n\bar u(t)\big|_{\partial\Omega}=R(t)\tau(s),\ \ \curl(\bar u(t))\big|_{\partial\Omega}=R(t)
$$
and from \eqref{3.vor0}, we get $\partial_n\bar u\big|_{\partial\Omega}\equiv0$
and finish the proof of the lemma.
\end{proof}
We are now ready  to complete the proof of the theorem. Indeed, due to the lemma, under the assumptions of the theorem, the function $\bar u$ satisfies the linear second order parabolic system \eqref{3.3} with {\it overdetermined} boundary conditions
$$
\bar u\big|_{\partial\Omega}=\partial_n\bar u\big|_{\partial\Omega}=0
$$
and must vanish due to the  proper Carleman estimate. In particular, we may use the following observability estimate established in \cite{Zu}.
\begin{proposition} Let $\tilde{\Omega}\subset\R^n$ be a bounded domain with smooth boundary and $\omega\subset\subset\tilde{\Omega}$. Let $\varphi\in C(0,T;[L^2(\tilde{\Omega})]^k)$ be a weak solution of the following linear parabolic system:
\begin{equation}\label{3.par}
\Dt\varphi-\nu\Dx\varphi=a_0(t,\mathtt x)\varphi+\sum_{i=1}^na_i(t,\mathtt x)\partial_{x_i}\varphi,\ \ \varphi\big|_{\partial\tilde{\Omega}}=0,\ \ \varphi\big|_{t=0}=\varphi_0,
\end{equation}
where the diffusion matrix $\nu=\nu \operatorname{Id}$, $\nu>0$ and $k\times k$ matrices $a_i$ satisfy
\begin{equation}\label{3.cond}
a_i\in L^\infty((0,T)\times\tilde{\Omega}), \ \ i=1,\cdots, n; \ \ a_0\in L^\infty(0,T;L^p(\tilde{\Omega})),\ \ p\ge n.
\end{equation}
Then, the following estimate holds for every $T>0$:
\begin{equation}\label{3.obs}
\|\varphi(T)\|_{L^2(\tilde{\Omega})}\le C_T\int_0^T\|\varphi(t)\|^2_{L^2(\omega)}\,dt,
\end{equation}
where the constant $C_T$ depends on $T$, $\tilde{\Omega}$, $\omega$ and the norms of $a_i$ used in \eqref{3.cond}, but is independent of the solution $\varphi$.
\end{proposition}
We want to apply this result to equations \eqref{3.3}. Note, first of all that $u_1(t)$ and $u_2(t)$ belong to $H^2(\Omega)$ for every fixed $t\in\R$, so conditions \eqref{3.cond} are satisfied. The next standard observation is that, since $\bar u$ satisfies zero Dirichlet and Neumann boundary conditions, we may extend $\bar u(t)$ by zero through the boundary $\partial\Omega$ (or through the connected component of the boundary which contains $\mathtt x_0$). Then, we get the extension $\varphi(t)$ of the function $\bar u(t)$ to a bigger domain $\tilde{\Omega}$ such that the interior of $\tilde{\Omega}\setminus\Omega$ is not empty and, therefore, contains a subdomain $\omega$. Note that $\varphi(t)$ thus defined solves \eqref{3.par} in $\tilde{\Omega}$ and $\varphi(t)\big|_\omega\equiv0$. The observability estimate \eqref{3.obs} ensure us that $\varphi(t)\equiv0$. Thus, $u_1(t)\equiv u_2(t)$ for all $t\in\R$ and the theorem is proved.
\end{proof}

\section{Concluding remarks and open problems}\label{s4}
In this section, we briefly discuss the obtained results and related open problems. We start with the most challenging and most interesting one.
\begin{remark} As we have already mentioned, in the case of generically chosen $g$ (i.e., when the number of solutions of the corresponding stationary Navier--Stokes problem is finite), there exists a single determining functional, see \cite{KKZ}. The key question here is whether or not this functional can be chosen depending on pressure only (i.e., in the form of $F(u)=\tilde F(P(u))$)?
\par
Since the proof of the result from \cite{KKZ} is based on the Takens delay embedding theorem, one of the possible way to tackle this problem is to consider the generalizations of the Takens theorem to the case where the observables are allowed to be taken only from some subspace/manifold in the space of all functionals (like $F(P(u))$ in our case). Unfortunately, we are not aware about any results in this direction which may help to solve the problem.
\par
Alternatively, we may try to rewrite the Navier--Stokes equation in terms of an evolutionary problem for pressure and  to apply the Takens theorem directly to this problem. To this end, we need to clarify the possibility to restore the velocity vector field $u(t,\mathtt x)$ if the pressure field $p(t,\mathtt x)$ is known. Of course, this question has a substantial independent interest since it would allow to present the classical Navier--Stokes system in an essentially different way, see \cite{R} for more details.
\par
Let us assume for simplicity that $\Omega$ is simply connected. Then the velocity field $u$ can be presented in the form $u=(\partial_{y}\Psi,-\partial_{x}\Psi)=:\nabla^{\bot}\Psi$, where $\Psi$ is a stream function which satisfies the boundary conditions
\begin{equation}\label{3.psi}
\Psi\big|_{\partial\Omega}=\partial_n\Psi\big|_{\partial\Omega}=0.
\end{equation}
Moreover, in the stream function formulation, equation \eqref{1.P} transforms to the Monge--Amp\'ere equation for $\Psi$:
\begin{equation}\label{3.monge}
\partial_{xx}\Psi\partial_{yy}\Psi-(\partial_{xy}\Psi)^2=\Dx p
\end{equation}
Thus, the problem is reduced to the uniqueness problem for  Monge--Amp\'ere equation. Unfortunately, it is known that this equation is usually of a mixed type, see e.g. \cite{R}, so we cannot use the classical theory which requires the ellipticity assumption $\Dx p\ge0$, see \cite{MA}. Instead, we may consider only smooth solutions and also we have overdetermined boundary conditions \eqref{3.psi} which a priori may help, but it is not clear how to use this extra information.
\end{remark}
\begin{remark} The next natural question is whether the extra vorticity functional $\Omega(u)$ is really necessary. Surprisingly, our technique allows us to remove this functional in the case where the boundary $\partial\Omega$ is not smooth and have edges. For instance, let $\Omega$ be a polygonal domain. Then, due to the analogue of Lemma \ref{Lem1.2} without extra assumption \ref{3.vor0}, we have
\begin{equation}\label{3.bound}
\Nx \bar u(t)\big|_{\partial\Omega}=R(t)\tau(s),
\end{equation}
 so $\Nx \bar u(t)\big|_{\partial\Omega}\in H^{1/2-\eb}(\partial\Omega)$ for small positive $\eb$. On the other hand, differentiating system \eqref{3.3} in $\mathtt x$ and using the $H^s$-regularity theorem ($s\le 3/2$) for solutions of the Laplace equations in polyhedral domains, see e.g. \cite{NP}, one can easily check that $\varphi(t):=\Nx\bar u(t)\in H^{3/2-\eb}(\Omega)\subset C(\bar\Omega)$.
 However, this continuity contradicts \eqref{3.bound} unless $R(t)\equiv0$.
\par
The above arguments are strongly based on the fact that the tangent vector $\tau(s)$ is discontinuous at the edge points and will not work if the domain $\Omega$ is smooth where we need the extra vorticity functional. On the other hand, we believe that this assumption is technical and the result should remain true without it in the case of smooth domains as well.
\par
One more natural open question is what happens in the cases where the boundary is empty, i.e., in the case of periodic boundary conditions. Recall that in the stream function formulation the Navier--Stokes system reads
\begin{equation}\label{3.str}
\Dt\Dx\Psi-J(\Psi,\Dx\Psi)=\nu\Dx^2\Psi+\curl g,\ \ \Psi\big|_{\partial\Omega}=\partial_n\Psi\big|_{\partial\Omega}=0,
\end{equation}
where $J(\Psi,\Dx\Psi);=\partial_x\Psi\partial_y\Dx\Psi-\partial_y\Psi\Dx\Psi$.  Therefore, if the pressure field $p(t)$ is given, this equation together with the Monge--Amp\'ere equation \eqref{3.monge}, give us the overdetermined system to be solved in order to find the stream function $\Psi$. Our approach (translated to the stream function language) is to verify that $\curl u\big|_{\partial\Omega}=\Dx\Psi\big|_{\partial\Omega}=0$ and to get a parabolic equation/system with overdetermined boundary conditions. This approach does not use the Monge--Amp\'ere equation inside of $\Omega$ at all and clearly can not be applied if the boundary is empty. It would be interesting to study the overdetermined system \eqref{3.str} and \eqref{3.monge} for the case of periodic boundary conditions and to get the analogous results for the possibility to reconstruct velocity from pressure.
\end{remark}
\begin{remark} To conclude, we mention one more interesting connection of the proved result with the trajectory approach to the study of evolutionary equations without uniqueness developed in \cite{CV02}, see also \cite{Z23} and references therein. For simplicity, we assume that the domain $\Omega$ is polygonal, so we do not need  the additional vorticity functional to restore the velocity from pressure.  Following the general theory, we define the trajectory phase space $K^+_u$ for the Navier--Stokes problem \eqref{1.Pro} as the set of all solutions of it belonging to the space $C(\R_+,H\cap H^2(\Omega)\cap H^1_0(\Omega))$:
$$
K^+_u:=\{u\in C(\R_+,H\cap H^2(\Omega)\cap H^1_0(\Omega)),\ u \text{ solves \eqref{1.Pro}}\}.
$$
Then, the semigroup $T_h$, $h\in\R_+$, of temporal shifts ($(T_hu)(t):=u(t+h)$) naturally acts on $K^+_u$ and define the {\it trajectory} dynamical system $(T_h,K^+_u)$ associated with the initial Navier--Stokes problem. Since in 2D case the solution is unique, the solution map $S:H\cap H^2(\Omega)\cap H^1_0(\Omega)\to K^+_u$ will be one-to-one. Moreover, if we endow the space $K^+_u$ by the local in time topology (the topology of $C_{loc}([0,\infty),H\times H^2(\Omega)\times H^1(\Omega))$, the map $S$ becomes a homeomorphism. Thus, in this case
$$
T_h=S\circ S(h)\circ S^{-1},\ \ T_h: K^+_u\to K^+_u,
$$
where $S(t)$ is the standard solution semigroup associated with the Navier--Stokes equation. Thus, the trajectory dynamical system $(T_h,K^+_u)$ is homeomorphic to the standard one and, since $S(t)$ possesses a global attractor $\Cal A$, the associated trajectory dynamical system also possesses a global attractor
$$
\Cal A^{tr}_u=S\Cal A
$$
 which is usually referred as a {\it trajectory} attractor associated with \eqref{1.Pro}.
\par
We now pass to the pressure component and define $K^+_p:=P(K^+_u)$ where $P(u)$ is defined in \eqref{1.2P}. Roughly speaking, $K^+_p$ consists of all pressure components of semitrajectories generated by \eqref{1.Pro}. Since we do not know whether the operator $P$ is one-to-one, we cannot define the standard evolution semigroup for the evolution of the pressure component, but, analogously to problems without uniqueness, we may define the associated trajectory dynamical system $(T_h,K^+_p)$ and verify the existence of an attractor $\Cal A^{tr}_p$ which will be referred as the trajectory attractor for the pressure component.
\par
Note that the observability estimate \eqref{3.obs} allows to establish the uniqueness on a semi-interval $t\in\R_+$ as well, so we conclude that the map $P:K^+_u\to K^+_p$ is one-to-one. Finally, utilizing the compactness, we establish that the restriction of $P$ to the trajectory attractor $\Cal A^{tr}_u$ gives us a homeomorphism  $\Cal A\sim\Cal A^{tr}_u\sim\Cal A^{tr}_p$. It would be interesting to understand whether and how the trajectory attractors theory may help to tackle the problems related with determining functionals.
\end{remark}

\end{document}